\RequirePackage[l2tabu, orthodox]{nag}
\documentclass[preprint, 11pt, reqno]{amsart}

\usepackage{mathrsfs}
\usepackage{fullpage}
\usepackage[dvipsnames]{xcolor}
\usepackage{graphicx}

\usepackage{palatino}

\def\R{\mathbb{R}}

\def\Z{\mathbb{Z}}

\renewcommand{\phi}{{\varphi}}

\def\<{\langle}
\def\>{\rangle}

\theoremstyle{plain} 
    \newtheorem{theorem}{Theorem}
    \newtheorem{lemma}[theorem]{Lemma}
    \newtheorem{proposition}[theorem]{Proposition}

\theoremstyle{definition} 

        \newtheorem{remark}{Remark}

\newcommand{\benu}{\begin{enumerate}\setlength\itemsep{4pt}}
\begin{document}

\title{How Many Modes Can a Mixture of Gaussians with Uniformly Bounded Means Have?}
\author[N. Kashyap]{Navin Kashyap}
\address{Navin Kashyap, Department of Electrical Communication Engineering,  Indian Institute of Science, Bangalore, India.}
\email{nkashyap@iisc.ac.in}

\author[M. Krishnapur]{Manjunath Krishnapur}
\address{Manjunath Krishnapur, Department of Mathematics, Indian Institute of Science, Bangalore, India.}
\email{manju@iisc.ac.in}

\thanks{N.K.\ is partially supported by the SERB MATRICS grant MTR/2017/000368. M.K.\ is partially supported by  UGC Centre for Advanced Study and the SERB MATRICS grant MTR/2017/000292.}

\begin{abstract}
We show,  by an explicit construction, that a mixture of univariate Gaussian densities with variance $1$ and means in $[-A,A]$ can have $\Omega(A^2)$ modes. This disproves a recent conjecture of Dytso, Yagli, Poor and Shamai~\cite{DYPS20} who showed that such a mixture can have at most $O(A^{2})$ modes and surmised that the upper bound could be improved to $O(A)$.  Our result holds even if an additional variance constraint is imposed on the mixing distribution. Extending the result to higher dimensions, we exhibit a mixture of Gaussians in $\R^{d}$, with identity covariances  and means inside $[-A,A]^{d}$, that has $\Omega(A^{2d})$ modes.
\end{abstract}

\maketitle

\section{Introduction}

A problem of long-standing interest in information theory is that of determining the capacity of the amplitude-constrained additive Gaussian noise channel \cite{Smith69},\cite{Smith71}. This is the channel defined by the input-output relationship
$
Y = X + Z
$
where the input $X$ is a random variable taking values in the interval $[-A,A]$ for some real number $A > 0$, and $Z \sim \mathcal{N}(0,1)$ is a standard Gaussian random variable independent of $X$. The capacity of this channel is defined as
\begin{equation}
C(A) = \max_{P_X: \text{supp}(P_X) \subseteq [-A,A]} I(X;Y),
\label{def:cap}
\end{equation}
where $I(X;Y)$ is the mutual information between the input $X \sim P_X$ and the output, $Y$, of the channel, and the maximum is taken over probability distributions $P_X$ supported within $[-A,A]$. Smith \cite{Smith69},\cite{Smith71} showed that the capacity-achieving input distribution $P_X$ is unique, symmetric about $0$, and, rather remarkably, has only finitely many mass points. However, his methods were insufficient to determine the precise form of the optimal input distribution, a question that has remained open since then. In particular, the number, $N^*(A)$, of mass points in the support of the optimal $P_X$ is still unknown.

In a recent breakthrough, Dytso, Yagli, Poor and Shamai \cite[Theorem~1]{DYPS20} showed that $N^*(A) = \Omega(A)$ and $N^*(A) = O(A^2)$.\footnote{For functions $f$ and $g$ of the positive real variable $A$, we write $g(A) = \Omega(f(A))$ to mean that there exists a constant $\kappa_0 > 0$ such that $g(A) \ge \kappa_0 f(A)$ for all $A > 0$. Similarly, $g(A) = O(f(A))$ means that there exists a constant $\kappa_1 > 0$ such that $g(A) \le \kappa_1 f(A)$ for all $A > 0$. Finally, $g(A) = \Theta(f(A))$ means that $g(A) = \Omega(f(A))$ and $g(A) = O(f(A))$ are simultaneously true.} Their $O(A^2)$ upper bound, in particular, was derived by relating the problem to one of counting the number of modes in a mixture of Gaussian densities with variance $1$ and means in $[-A,A]$, as we briefly explain next. But before doing so, we would like to clarify that we use ``modes'' to refer exclusively to the local maxima of a density function obtained as a mixture (i.e., a convex combination) of Gaussian density functions. In particular, we do \emph{not} consider the local maxima of the likelihood function of a Gaussian mixture model, which is an entirely different problem set-up --- see e.g.\ \cite{JZBWJ16}.

Suppose that $X^*$ is an input random variable distributed according to the optimal (capacity-achieving) distribution $P^*$ for the problem in \eqref{def:cap}. By the result of Smith \cite{Smith71}, $P^*$ is of the form $p_1\delta_{a_1}+\ldots +p_N\delta_{a_N}$, where $-A\le a_1<a_2<\ldots <a_N\le A$, and the probability masses $p_k>0$ sum to $1$. So, the corresponding output random variable $Y^* = X^* + Z$ has density $f_{Y^*}(t)=\sum_{k=1}^Np_k\phi(t-a_k)$, where $\phi(t)=\frac{1}{\sqrt{2\pi}}e^{-\frac12 t^2}$. In other words, the density $f_{Y^*}$ is a mixture of Gaussian densities with unit variance and centres (means) $a_k$ uniformly bounded within $[-A,A]$. Dytso et al.\ \cite{DYPS20} showed, by appealing to a certain total positivity property of the Gaussian kernel, that $N^*(A)$ is at most twice the number of modes\footnote{Explicitly, Theorem~5 and Lemma~3 in \cite{DYPS20} show that $N^*(A)$ is at most one more than the number of local extrema of $f_{Y^*}$. But since $f_{Y^*}$ is a real-analytic function that decays to $0$ outside $[-A,A]$, the number of its modes (local maxima) exceeds the number of its local minima by exactly $1$, as can be seen from the representative plot of the function in Fig.~\ref{fig:gaussmix}. Hence, one more than the number of its local extrema is exactly equal to twice the number of its modes.}
of the density $f_{Y^*}$. Subsequently, the $O(A^2)$ bound on $N^*(A)$ was obtained via a complex-analytic method (a variant of Jensen's formula) of bounding the number of zeros of the derivative $f_{Y^*}'$. 

The techniques of Dytso et al.\ only use the fact that $f_{Y^*}$ is a mixture of variance-$1$ Gaussians with means in $[-A,A]$. Thus, they effectively obtained the following result: 
\begin{quote}
Let $\mathsf{m}(A)$ denote the maximum number of modes that any mixture of Gaussian densities with variance $1$ and means in $[-A,A]$ can have. Then, $\mathsf{m}(A) = O(A^2)$.  
\end{quote}
Dytso et al.\ further conjectured \cite[Remark~9]{DYPS20} that $\mathsf{m}(A) = O(A)$, and as a consequence, $N^*(A)=O(A)$. In fact, since they prove that $N^*(A) = \Omega(A)$, their conjecture would have implied that $N^*(A)=\Theta(A)$, thus answering a 50-year-old open question.

The main aim of this paper is to give a proof of the following proposition, which disproves their conjecture on $\mathsf{m}(A)$.

\begin{proposition} 
$\mathsf{m}(A) = \Omega(A^2)$, i.e., $\mathsf{m}(A) \ge c_0 A^2$ for some constant  $c_0>0$ and all $A>0$.
\label{prop:1}
\end{proposition}

Thus, in conjunction with the result of Dytso et al., we have that $\mathsf{m}(A) = \Theta(A^2)$. While our result effectively blocks this particular route to proving that $N^*(A) = \Theta(A)$, numerical work does indeed suggest that this is the right order of growth of $N^*(A)$ with $A$.

\begin{remark}
Independently of us, but subsequent to us, Polyanskiy and Wu \cite{PW20} have obtained a slightly weaker result that also essentially disproves the conjecture of Dytso et al. They give an example of a random variable $X$ having a \emph{density} $\pi$ supported within $[-A,A]$ such that the density, $\pi*\phi$, of $X+Z$ has $\Omega(A^2)$ modes \cite[Section~5.4]{PW20}. We, on the other hand, construct a \emph{discrete} random variable $X$ such that the density of $X+Z$ has $\Omega(A^2)$ modes. Indeed, our distribution on $X$ consists of a finite number of equally-weighted and equally-spaced point masses, and it is interesting to see that such a simple mixing distribution suffices to generate an order-optimal number of modes. By way of contrast, the density $\pi$ considered by Polyanskiy and Wu has a sinusoidal shape.
\end{remark}

The result of Proposition~\ref{prop:1} does not change qualitatively if we further impose a variance constraint on the mixing distribution. To be precise, consider now Gaussian mixtures $f_Y(t)=\sum_{k=1}^Np_k\phi(t-a_k)$, with centres $a_k$ again constrained to be in $[-A,A]$, but additionally requiring the random variable $X \sim  \sum_{k=1}^N p_k \delta_{a_k}$  to have variance $\mathrm{var}(X) \le 1$. (Of course, any constant bound on the variance will do; we take the bound to be $1$ for simplicity.) Addition of a variance constraint is motivated by considerations similar to those outlined above, but now arising from the study of the amplitude- and variance-constrained capacity of the additive Gaussian noise channel \cite{DYPS20},\cite{Smith71}:
$$
C(A,\sigma^2) = \max_{P_X: \text{supp}(P_X) \subseteq [-A,A] \atop \text{var}(P_X) \le \sigma^2} I(X;Y),
$$

Let $\mathsf{m}_{\#}(A)$ denote the maximum number of modes among densities $f_Y$ obtained as mixtures of Gaussian densities with variance $1$ and means in $[-A,A]$, where additionally the mixing distribution has variance bounded above by $1$. We then have the following result.

\begin{proposition}
$\mathsf{m}_{\#}(A) = \Omega(A^2)$, i.e., $\mathsf{m}_{\#}(A) \ge c_{\#} A^2$ for some constant  $c_{\#}>0$ and all $A>0$.
\label{prop:2}
\end{proposition}

Note that this result is stronger than that of Proposition~\ref{prop:1}: we clearly have $\mathsf{m}(A) \ge \mathsf{m}_{\#}(A)$, as the latter is a maximum over a more constrained set of densities.
 
Our results carry over to higher dimensions without substantial change. In this paper, we extend the use of the notation $| \cdot |$ to include the $\ell^\infty$-norm on $\R^d$: for a vector $x = (x_1,\ldots,x_d) \in \R^d$, $|x| = \max_i |x_i|$. Let $\phi_{d}$ denote the standard Gaussian density (zero mean and identity covariance) in $\R^{d}$.  Let $\mathsf{m}_{d}(A)$ denote the maximum number of modes that the Gaussian mixture density $f(t)=p_{1}\phi_{d}(t-a_{1})+\ldots +p_{N}\phi_{d}(t-a_{N})$ can have, subject to the constraints that $|a_k|\le A$ for all $i$, and $p_k>0$ sum to $1$. 
\begin{proposition}\label{prop:3} With the above notation, $\mathsf{m}_{d}(A)\ge c_d A^{2d}$ for a constant $c_d>0$ and all $A > 0$. 
\end{proposition}
However, we are not aware of a corresponding upper bound. It is worth remarking here that, besides the information-theoretic considerations that motivate us, there is considerable mathematical interest in counting modes of Gaussian mixtures --- cf.\ \cite{AEH19}, \cite{CW03}, and the references therein. For instance, it was conjectured by Sturmfels (see \cite[Conjecture~5]{AEH19}) that a Gaussian mixture (with identity covariance matrices, as we have taken) with $N$ components in $\R^d$ has at most $\binom{N+d-1}{d}$ modes. In one dimension, this bound reduces to $N$, a fact first proved in \cite{Sil81} --- see also \cite[Section~2.4]{CW03}. These studies do not put any constraint on the means of the Gaussians involved in the mixture, while the uniform boundedness of the means is a key feature of our results.

\subsection*{\emph{A sketch of the proofs.}} The main ingredients in our proofs of Propositions~\ref{prop:1} and \ref{prop:2} are mixtures of the form 
\begin{align}
\gamma_{a,N}(x) := \frac{1}{2N+1} \sum_{n=-N}^N \phi(x-an),
\label{def:gamma}
\end{align}
with $a > 0$. This is an equally-weighted  mixture of $2N+1$ Gaussians with equally-spaced centres (means) $an$, for integers $n$ between $-N$ and $N$. Fig.~\ref{fig:gaussmix} illustrates the shape of the unnormalized mixture 
\begin{align}
f_{a,N}(x) := \sum_{n=-N}^N \phi(x-an).
\label{def:faN}
\end{align}

\begin{figure}[t]
\centering
\includegraphics[width=0.8\textwidth]{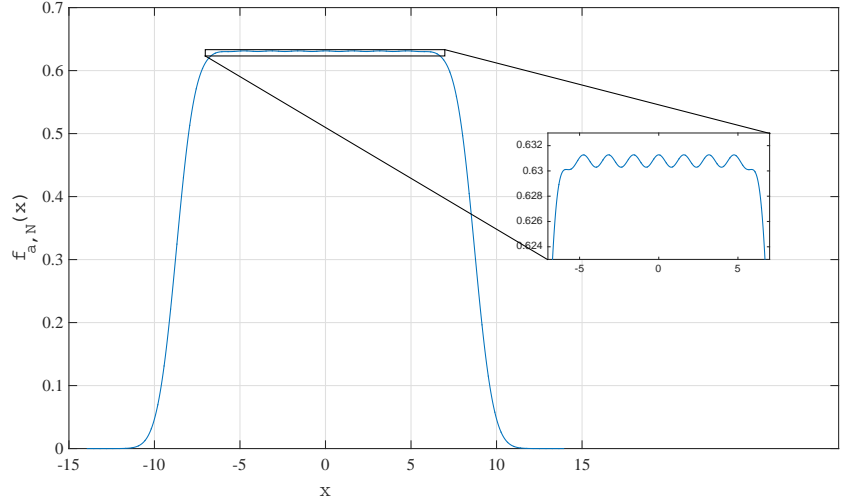}
\caption{A plot of $f_{a,N}(x) = \sum_{n=-N}^N \phi(x-an)$ for $N = 5$ and $a = 2\sqrt{\pi/N}$.}
\label{fig:gaussmix}
\end{figure}

We will show that by choosing $a = \frac{c}{\sqrt{N}}$ for a suitable constant $c > 0$, the resulting unnormalized mixture $f_{a,N}$ has centres in $[-c\sqrt{N},c\sqrt{N}]$ and at least $N-1$ modes. Since scaling by a constant has no effect on the number of modes, the same holds for the mixture $\gamma_{a,N}$, which suffices to prove Proposition~\ref{prop:1}. The proof is elaborated in Section~\ref{sec:proofoffirstsecondproposition}.

For Proposition~\ref{prop:2}, we work with the mixture 
\begin{align}
\Gamma_{\alpha; \, a,N}(x) & := (1-2\alpha) \, \phi(x) + \alpha \, \gamma_{a,N}(x+2aN) + \alpha \, \gamma_{a,N}(x-2aN) \notag \\
& = (1-2\alpha) \, \phi(x) + \frac{\alpha}{2N+1} \sum_{n=-3N}^{-N} \phi(x-an) + \frac{\alpha}{2N+1} \sum_{n=N}^{3N} \phi(x-an),
\label{def:Gamma}
\end{align}
where $a = \frac{c}{\sqrt{N}}$ is as above, and $\alpha \in (0,\frac12)$. This is a Gaussian mixture with centres at $0$ and $\pm an$, $n = N,N+1,\ldots,3N$, weighted by $1-2\alpha$ and $\frac{\alpha}{2N+1}$, respectively. It is easy to check that by taking $\alpha \sim \frac{1}{N}$, we can get the underlying random variable $X$ to have variance at most $1$. We will, moreover, show that for this choice of $\alpha$, the mixture $\Gamma_{\alpha; \, a,N}$ has $\Omega(N)$ modes. Since $\Gamma_{\alpha; \, a,N}$ has all its centres within $[-3c\sqrt{N},3c\sqrt{N}]$, this will prove Proposition~\ref{prop:2}. The detailed proof is in Section~\ref{sec:proofoffirstsecondproposition}.

The proof of Proposition~\ref{prop:3} is entirely analogous to that of Proposition~\ref{prop:1}, and uses a mixture with equal weights and centers at $ak$, where $k=(k_{1},\ldots ,k_{d})\in \Z^{d}$ with $-N\le k_{i}\le N$, for  appropriately chosen $a$ and $N$ (the right choices turn out to be $a=1/A$ and $N=A^{2}$).  Details are in Section~\ref{sec:proofofthirdproposition}.

\section{Proof of Proposition~\ref{prop:1} and Proposition~\ref{prop:2}}
\label{sec:proofoffirstsecondproposition}
As mentioned in the proof sketch above, Proposition~\ref{prop:1} is proved by considering equal-weighted mixtures $\gamma_{a,N}$, as defined in \eqref{def:gamma}, for a suitable choice of $a$.
Our analysis is based on the fact that, for any $a > 0$, the unnormalized mixture $f_{a,N}$ defined in \eqref{def:faN} is a truncation of the infinite series 
\begin{align*}
f_a(x):=\sum_{n\in \Z}\phi(x-an).
\end{align*}
Note that $f_a$ is well-defined and periodic with period $a$. By standard real-analysis arguments, $f_a$ is continuous on $\R$.

We first obtain an estimate for $h_a := f_a(0) - f_a(\frac{a}{2})$, which we will use in our proofs. 
\begin{lemma}
For any $a > 0$, we have
\begin{align*}
\frac{4}{a}e^{-\frac{2\pi^2}{a^2}} \ \le \  h_a \ \le \ \frac{4}{a}e^{-\frac{2\pi^2}{a^2}} \, {\left(1-e^{-\frac{2\pi^2}{a^2}}\right)}^{-1}. 
\end{align*}
\label{lem:ha}
\end{lemma}
\begin{proof}
We prove the lower bound first. By the Poisson summation formula\footnote{With the notation $\hat{f}(\lambda)=\int f(x) \, e^{-2\pi i \lambda x}dx$, we have $\sum\limits_{n\in \Z}f(x+n)=\sum\limits_{n\in \Z}\hat{f}(n)e^{2\pi i nx}$.}, for any $x \in \R$,
\begin{equation}
\label{eq:Poisson}
f_a(x)=\sum\limits_{n\in \Z}\phi\biggl(a\bigl(\frac{x}{a}-n\bigr)\biggr) = \frac{1}{a}\sum\limits_{n\in \Z} e^{-\frac{2\pi^2n^2}{a^2}}e^{2\pi i n \frac{x}{a}},
\end{equation}
from which we get
$$f_a(0) \ = \ \frac{1}{a} \sum\limits_{n\in \Z} e^{-\frac{2\pi^2n^2}{a^2}} > \frac{1}{a} \ > \ 
\frac{1}{a} \sum\limits_{n\in \Z} (-1)^n \, e^{-\frac{2\pi^2n^2}{a^2}} \ = \ f_a\left(\frac{a}{2}\right).$$
In particular, we have 
$$
h_a \ = \ f_a(0) -f_a\left(\frac{a}{2}\right) 
\ = \ \frac{2}{a} \sum\limits_{n \in \Z, \atop n \text{ odd}} e^{-\frac{2\pi^2n^2}{a^2}} 
\ = \ \frac{4}{a} \sum\limits_{n > 0, \atop n \text{ odd}} e^{-\frac{2\pi^2n^2}{a^2}} 
\ > \ \frac{4}{a} e^{-\frac{2\pi^2}{a^2}}.
$$

For the upper bound, consider
\begin{align*}
\left|f_a(x)-\frac{1}{a}\right| \ \le \ \frac{1}{a} \sum_{n\not=0}  e^{-\frac{2\pi^2n^2}{a^2}} 
\ \le \ \frac{2e^{-\frac{2\pi^2}{a^2}}}{a\left(1-e^{-\frac{2\pi^2}{a^2}}\right)},
\end{align*}
the first inequality arising from \eqref{eq:Poisson}, and the second inequality being obtained by replacing $n^2$ by $n$ to get a geometric series. Thus,
\begin{align}
h_a \ = \ \left|f_a(0)-\frac{1}{a}\right| + \left|f_a\left({\frac{a}{2}}\right) - \frac{1}{a}\right| \ \le \ \frac{4e^{-\frac{2\pi^2}{a^2}}}{a\left(1-e^{-\frac{2\pi^2}{a^2}}\right)},
\label{ha_upbnd}
\end{align}
which is the claimed upper bound.
\end{proof}

Thus, for $a \ll 1$, we have $h_a \approx \frac{4}{a} \exp(-\frac{2\pi^2}{a^2})$. We actually need only the lower bound on $h_a$ for our arguments. 

\begin{remark}
A minor modification in the above proof shows that the bounds in Lemma~\ref{lem:ha} in fact apply to $\overline{h}_a = \max(f_a) - \min(f_a)$ as well. Indeed, the lower bound is obvious, since $\overline{h}_a \ge h_a$. For the upper bound, we observe that if $x^*$ and $x_*$ achieve the maximum and minimum, respectively, of $f_a$, then $\overline{h}_a = |f_{a}(x^*)-\frac1a| + |f_{a}(x_*)-\frac1a|$, so that the upper bound in \eqref{ha_upbnd} still holds.
\end{remark}

It is clear from \eqref{eq:Poisson} that  $f_{a}(0)>f_{a}(x)$ for all $x\in [-\frac{a}2,\frac{a}2]$, since there is non-trivial cancellation in the terms of the series unless $x$ is an integer multiple of $a$. By the fact that $f_a$ has period $a$, we see that $na$ is a strict maximum of $f_{a}$ in the interval $I_{a,n} := [na-\frac{a}2,na+\frac{a}2]$ for any $n \in \Z$. We wish argue that $f_{a,N}$ also has local maxima within those intervals $I_{a,n}$ that are contained in $[-\frac12 aN,\frac12 aN]$. For this, we will need the simple lemma stated next.



\begin{lemma}
Let $g$ be a continuous function such that $|f_a - g| < \frac12 h_a$ on a subset $S \subseteq \R$. Then, $g$ has a local maximum in the interior of any interval $I_{a,n}$ that is contained within $S$.
\label{lem:perturb}
\end{lemma}
\begin{proof}
Recall that $I_{a,n} = [na-\frac{a}2,na+\frac{a}2]$, for $n \in \Z$. If $|f_a - g| < \frac12  h_a$ holds on $I_{a,n}$, then we have 
\begin{align*}
g(na) - g(na-{\textstyle \frac{a}2)} & = \bigl(g(na)-f_a(na)\bigr) + \bigl(f_a(na)-f_a(na-{\textstyle \frac{a}2} )\bigr)  + \bigl(f_a(na-{\textstyle \frac{a}2} )-g(na-{\textstyle \frac{a}2} )\bigr) \\
& > \bigl(-\frac12 h_a\bigr) + h_a + \bigl(-\frac12 h_a\bigr) \\
& = 0.
\end{align*}
Hence, $g(na) > g(na-\frac{a}2)$. Analogously, $g(na) > g(na+\frac{a}2)$. Therefore, the global maximum of $g$ in $I_{n,a}$ is attained at and interior point. In particular, $g$ has a local maximum strictly between $na-\frac{a}2$ and $na+\frac{a}2$.
\end{proof}

%

We now have the facts necessary to furnish proofs of Propositions~\ref{prop:1} and \ref{prop:2}.

\begin{proof}[Proof of Proposition~\ref{prop:1}]
We apply Lemma~\ref{lem:perturb} with $g=f_{a,N}$. Note first that 
\begin{align*}
|f_a(x)-f_{a,N}( x)| &= \frac{1}{\sqrt{2\pi}} \sum_{n:|n|>N}e^{-\frac12 (an-x)^2} \\
&\le \frac{1}{\sqrt{2\pi}} \sum_{n:|n|>N}e^{-\frac12 (a|n|-|x|)^2} \;\;\; \bigl(\mbox{since } |an-x| \ge \bigl|a|n|-|x|\bigr|\bigr) \\
&= \frac{2}{\sqrt{2\pi}}  \sum_{n>N}e^{-\frac12 (an-|x|)^2} \\
&= \frac{2}{\sqrt{2\pi}}  \, e^{-\frac12(aN-|x|)^2} \sum_{n>N}e^{-\frac{1}{2}a(n-N)(a(N+n)-2|x|)} \\
&\le \frac{2}{\sqrt{2\pi}}  \, e^{-\frac12(aN-|x|)^2} \sum_{n>N}e^{-a(n-N)(aN-|x|)} \; . 
\end{align*}
%
%
Now take $|x|\le \frac12aN$ to get 
\begin{align}
|f_a(x)-f_{a,N}( x)| &\le  \frac{2}{\sqrt{2\pi}} \, e^{-\frac18 a^2N^2} \sum_{n>N} e^{-\frac12a^2N(n-N)} \notag \\
&= {\frac{2}{\sqrt{2\pi}} \, e^{-\frac18 a^2N^2}} \, \frac{e^{-\frac12a^2N}}{1-e^{-\frac12a^2N}} \; .
\label{maxdev_bnd}
\end{align}

If we take $a=\frac{2\sqrt{\pi}}{\sqrt{N}}$ and $S=[-\frac12aN,\frac12aN] = [-\sqrt{\pi N},\sqrt{\pi N}]$, then \eqref{maxdev_bnd} holds for all $x\in S$, so that
\begin{equation}
|f_a(x)-f_{a,N}( x)| \ \le \ C_0 e^{-\frac12\pi N}
\label{ineq:C0}
\end{equation}
with $C_0 = \frac{2}{\sqrt{2\pi}}\left(\frac{e^{-2\pi}}{1-e^{-2\pi}}\right)$.
On the other hand, from the lower bound for $h_a$ in Lemma~\ref{lem:ha}, we have
$$
h_a \ge 2 \, \sqrt{\frac{N}{\pi}} \, e^{-\frac12 \pi N}.
$$

As $C_0 < \frac{2}{\sqrt{\pi}}\left(\frac{e^{-2\pi}}{1-e^{-2\pi}}\right)$, we have for all $N \ge 1$,
$C_0 e^{-\frac12\pi N} < \left(\frac{e^{-2\pi}}{1-e^{-2\pi}}\right) h_a$, and consequently,
\begin{align*}
|f_a(x)-f_{a,N}(x)| <  \left(\frac{e^{-2\pi}}{1-e^{-2\pi}}\right) h_a \ \ \text{ for all } x \in S.
\end{align*}
Since $\frac{e^{-2\pi}}{1-e^{-2\pi}} \approx 0.0019$, the conclusion of Lemma~\ref{lem:perturb} holds, i.e., $f_{a,N}$ has a local maximum in the interior of each of the intervals $I_{a,n}$ contained in $S = [-\frac12aN,\frac12aN]$. There are at least $N-1$ such intervals $I_{a,n}$, and hence, $f_{a,N}$ has at least $N - 1$ local maxima within $S$. Thus, we conclude that the Gaussian mixture $\gamma_{a,N} = \frac{1}{2N+1} f_{a,N}$ (with $a=\frac{2\sqrt{\pi}}{\sqrt{N}}$), which has all its centres inside $[-2\sqrt{\pi N},2\sqrt{\pi N}]$, has at least $N - 1$ modes (within $S = [-\sqrt{\pi N},\sqrt{\pi N}]$). Choosing $N=A^{2}$ proves Proposition~\ref{prop:1}. 
\end{proof}

\bigskip

\begin{proof}[Proof of Proposition~\ref{prop:2}]
Consider $\Gamma_{\alpha;a,N}$ as defined in \eqref{def:Gamma}, with $a = \frac{2\sqrt{\pi}}{\sqrt{N}}$ as in the proof of Proposition~\ref{prop:1}. This is the density of $Y = X+Z$, where $Z \sim \mathcal{N}(0,1)$ is independent of $X \sim (1-2\alpha) \delta_0 + \frac{\alpha}{2N+1} \sum_{n=N}^{3N} (\delta_{-an} + \delta_{an})$. We then have 
\begin{align*}
\mathrm{var}(X) &= \frac{\alpha}{2N+1} \sum_{n=N}^{3N} 2(an)^2  \\
& \le \frac{2\alpha a^2}{2N+1} \sum_{n=1}^{3N} n^2 \\
& = \frac{2\alpha a^2}{2N+1} \, \left(\frac{3N(3N+1)(6N+1)}{6}\right) \\
&\le \alpha a^2 (3N)(3N+1) \\
&= 12 \pi (3N+1) \alpha \ \ \ \ \ \ (\text{using $a = {\textstyle \frac{2\sqrt{\pi}}{\sqrt{N}}}$}).
\end{align*}
Hence, setting $\alpha = \frac{1}{12\pi(3N+1)}$, we obtain $\mathrm{var}(X) \le 1$.

We will next show that, with $a$ and $\alpha$ as above, $\Gamma_{\alpha;a,N}$ has $\Omega(N)$ modes. This suffices to prove the proposition, since $\Gamma_{\alpha;a,N}$ is a Gaussian mixture with all of its centres in $[-6 \sqrt{\pi N}, 6 \sqrt{\pi N}]$. 

It is easy to check that $\Gamma_{\alpha;a,N}$ has a mode at $0$. We will show that, when $N$ is sufficiently large, $\Gamma_{\alpha;a,N}$ has at least $N-1$ modes in each of the intervals $[-5\sqrt{\pi N}, -3\sqrt{\pi N}]$ and $[3\sqrt{\pi N}, 5\sqrt{\pi N}]$. By symmetry, it is enough to show this for the interval $[3\sqrt{\pi N}, 5\sqrt{\pi N}]$. For this, we use Lemma~\ref{lem:perturb} with $g = \bigl(\frac{2N+1}{\alpha}\bigr) \Gamma_{\alpha; a,N}$. For this choice of $g$, we have
\begin{align}
|f_a(x) - g(x)| &= \left| \sum\limits_{n: |n| < N \text{ or } |n| > 3N} \phi(x-an) - \left(\frac{1-2\alpha}{\alpha}\right) (2N+1) \phi(x)\right|  \notag \\
& \le \sum\limits_{n: |n| < N \text{ or } |n| > 3N} \phi(x-an) + \left(\frac{1-2\alpha}{\alpha}\right) (2N+1) \phi(x) \notag \\
& \le \sum\limits_{n < N \text{ or } n > 3N} \phi(x-an) + \left(\frac{2N+1}{\alpha}\right) \, \phi(x) .
\label{ineq:prop2proof}
\end{align}

Consider the first term in \eqref{ineq:prop2proof} above. Writing $x' = x-2aN$, we have
\begin{align*}
\sum\limits_{n < N \text{ or } n > 3N} \phi(x-an)  &= \sum\limits_{n < N \text{ or } n > 3N} \phi(x'-a(n-2N))  \\
&= \sum\limits_{n < -N \text{ or } n > N} \phi(x'-an) \\
& = |f_a(x') - f_{a,N}(x')|  \\
& \le \ C_0 e^{-\frac12\pi N}
\end{align*}
for $|x'| \le \frac12 aN$ and $C_0 = \frac{2}{\sqrt{2\pi}}\left(\frac{e^{-2\pi}}{1-e^{-2\pi}}\right)$, by \eqref{ineq:C0} in the proof of Proposition~\ref{prop:1}. Thus, for $|x-2aN| \le \frac12 aN$, i.e., for $x \in [3\sqrt{\pi N}, 5\sqrt{\pi N}]$, we see that the first term in \eqref{ineq:prop2proof} is bounded above by $C_0 e^{-\frac12\pi N}$.

Turning our attention to the second term in \eqref{ineq:prop2proof}, we first observe that $\frac{2N+1}{\alpha} \le C_0' N^2$ for some constant $C_0'$. Thus,
$$
\left(\frac{2N+1}{\alpha}\right) \, \phi(x)  \le C_0' N^2 \phi(x) \le \frac{1}{\sqrt{2\pi}} C_0' N^2 e^{-\frac92 \pi N},
$$
for $x \ge 3\sqrt{\pi N}$. 

Combining these bounds, we obtain that for $x \in [3\sqrt{\pi N}, 5\sqrt{\pi N}]$, 
$$
|f_a(x) - g(x)| \ \le \ C_0 e^{-\frac12\pi N} + \frac{1}{\sqrt{2\pi}} C_0' N^2 e^{-\frac92 \pi N} 
\ \le \ 2C_0 e^{-\frac12 \pi N}
$$
when $N$ is sufficiently large. As shown in the proof of Proposition~\ref{prop:1}, $C_0 e^{-\frac12 \pi N} < \left(\frac{e^{-2\pi}}{1-e^{-2\pi}}\right) h_a$. Consequently, when $N$ is sufficiently large, for $x \in [3\sqrt{\pi N}, 5\sqrt{\pi N}]$, we have
$$
|f_a(x) - g(x)| \ < \ \left(\frac{2e^{-2\pi}}{1-e^{-2\pi}}\right) h_a \ < \ 0.004 \, h_a.
$$
Then, applying Lemma~4, we obtain that, for all sufficiently large $N$, the function $g = \bigl(\frac{2N+1}{\alpha}\bigr) \Gamma_{\alpha; a,N}$ has at least $N-1$ modes within the interval $[3\sqrt{\pi N}, 5\sqrt{\pi N}]$. This naturally holds for $\Gamma_{\alpha; a,N}$ as well, thus proving the proposition. 
\end{proof}

\section{Proof of Proposition~\ref{prop:3}}
\label{sec:proofofthirdproposition}
Since the proof is entirely analogous to that of Proposition~\ref{prop:1}, we shall only sketch the modifications needed and omit the details. For $a>0$ and integer $N\ge 1$ and define the functions
\begin{align*}
f_{a}(x) &=\sum_{n\in \Z^{d}}\phi_{d}(x-na), \\
f_{a,N}(x) &=\sum_{n\in Q_{N}}\phi_{d}(x-na),
\end{align*}
where $Q_{N}=\{n\in \Z^{d}\; : \; -N\le n_{i}\le N\mbox{ for }1\le i\le d\}$.  By the Poisson summation formula on $\R^{d}$ with respect to the lattice $\Z^{d}$, we get
\begin{align*}
f_{a}(x) &=\frac{1}{a^{d}}\sum_{p\in \Z^{d}}e^{-\frac{1}{2a^{2}}|p|^{2}+\frac{2\pi i}{a}\<p,x\>} \\
&= \frac{1}{a^{d}}\biggl(1+2e^{-\frac{1}{2a^{2}}}\sum_{j=1}^{d}\cos(2\pi x_{j}/a) + O(e^{-\frac{2}{a^{2}}})\biggr) \, ,
\end{align*}
where the big-O term includes the contribution of all $p$ with $|p|\ge 2$. 
Since  $\cos(2\pi t)\le 1-8t^{2}$ for any $t\in \R$, we see that
when $|x|=\frac{a}{2}$, 
\begin{align*}
f_{a}(x) &\le \frac{1}{a^{d}}\biggl(1+2e^{-\frac{1}{2a^{2}}}\sum_{j=1}^{d}(1-{\textstyle \frac{8}{a^{2}}} x_{j}^{2}) +O(e^{-\frac{2}{a^{2}}}) \biggr)\\
&=\frac{1}{a^{d}}\bigl(1+2(d-2)e^{-\frac{1}{2a^{2}}}+O(e^{-\frac{2}{a^{2}}}) \bigr).
\end{align*}
Since $f_{a}(0)=\frac{1}{a^{d}}\bigl(1+2de^{-1/(2a^{2})}+O(e^{-2/a^{2}})\bigr)$, we see that 
\begin{align*}
f_{a}(0)-\sup_{|x|=\frac{a}{2}}f_{a}(x)&= \frac{1}{a^d}\bigl(4e^{-\frac{1}{2a^{2}}}+O(e^{-\frac{2}{a^{2}}})\bigr) \,,
\end{align*}
which is at least $h_{a}:= \frac{3}{a^d} e^{-\frac{1}{2a^{2}}}$, for small enough $a$.  By periodicity, in each cube of the form $na+[-\frac12a,\frac12a]^{d}$, the graph of $f_{a}$ has a hill with peak at $na$ and having  height at least $h_{a}$. Further,
\begin{align*}
|f_{a}(x)-f_{a,N}(x)| &= \frac{1}{(2\pi)^{\frac{d}{2}}}\sum_{n\in \Z^{d}\setminus Q_{N}}e^{-\frac{1}{2a^{2}}|(x+na)|^{2}} \\
&= O(e^{-\frac18 a^{2}N^{2}}) \;\;\; \mbox{ for }|x|\le \frac12 aN.
\end{align*}
Now take $a=\frac{c}{\sqrt{N}}$ to see that for suitable $c,c'$,
\begin{align*}
\sup\limits_{|x|\le c'\sqrt{N}}|f_{a}(x)-f_{a,N}(x)|< \frac12 h_{a}.
\end{align*}
 Therefore, the function $f_{a,N}$ has a local maximum in each cube of the form $na+[-\frac12 a,\frac12 a]^{d}$ that is contained inside the larger cube  $[-c'\sqrt{N},c'\sqrt{N}]^{d}$. This is because the perturbation is too small to wash away the local maximum of $f_{a}$ located at $na$. The number of such cubes is about $(2c'\sqrt{N}/a)^{d}$, which is $\Theta(N^{d})$. 

Taking $N=\sqrt{A}$ gives us a  function $f_{a,N}$ (with $a=c/A$) that is a mixture of Gaussians with centers in $Q_{A}$ and having $\Theta(A^{2d})$ modes. This was the claim of Proposition~\ref{prop:3}.

%

\section{A Concluding Remark} 
We note that our proof techniques do not extend immediately to the heteroscedastic setting, in which the different Gaussian components of the mixture may have different variances (or covariance matrices), while the means remain bounded within $[-A,A]^d$. This is because of our reliance on the Poisson summation formula to obtain some of the estimates needed in our proofs. While our results of course show that the maximum number of modes in the heteroscedastic setting is $\Omega(A^{2d})$, it is not clear, even in the $d=1$ case, if heteroscedasticity can help create an even larger number of modes. This could be an interesting direction of future research.

%
\section*{Acknowledgement}  The authors would like to thank Alex Dytso for asking whether imposing a variance constraint on the mixing distribution would influence the number of modes, which led to Proposition~\ref{prop:2}. They also thank Yury Polyanskiy and Yihong Wu for sharing an early draft of their manuscript \cite{PW20}. 

\bibliographystyle{plain}
\bibliography{references}

%
%
%
%
%

\end{document}